\documentclass[12pt]{amsart}
\usepackage[foot]{amsaddr}
\usepackage{cite}
\usepackage{lipsum}
\usepackage{amscd,amsmath,amssymb,amsthm,amsfonts}
\usepackage{mathtools,xparse}
\usepackage{latexsym}
\usepackage{enumerate}
\usepackage{mwe}
\usepackage[markcase=noupper
]{scrlayer-scrpage}
\newtheorem{theorem}{Theorem}[section]
\newtheorem{proposition}[theorem]{Proposition}
\newtheorem{corollary}[theorem]{Corollary}
\newtheorem{lemma}[theorem]{Lemma}

\theoremstyle{definition}
\newtheorem{definition}[theorem]{Definition} % definition numbers are dependent on theorem numbers
\newtheorem{example}[theorem]{Example} % same for example numbers
\newtheorem{remark}[theorem]{Remark}

\begin{document}
\title{On Riemannian manifolds\\ with positive weighted Ricci curvature\\ of negative effective dimension}
\date{\today}
\author{MAI, Cong Hung}
\address{Department of Mathematics, Kyoto University, Kyoto 606-8502, Japan} 
\email{hongmai@math.kyoto-u.ac.jp}

\begin{abstract}
		In this paper, we investigate complete Riemannian manifolds
		satisfying the lower weighted Ricci curvature bound $\mathrm{Ric}_{N} \geq K$ with $K>0$
		for the negative effective dimension $N<0$.
		We analyze two $1$-dimensional examples of constant curvature $\mathrm{Ric}_N \equiv K$
		with finite and infinite total volumes.
		We also discuss when the first nonzero eigenvalue of the Laplacian
		takes its minimum under the same condition $\mathrm{Ric}_N \ge K>0$,
		as a counterpart to the classical Obata rigidity theorem.
		Our main theorem shows that, if $N<-1$ and the minimum is attained,
		then the manifold splits off the real line as a warped product of hyperbolic nature.
	\end{abstract}
%%% AMS subject classification

	\maketitle

	%\tableofcontents
	
	\section{Introduction}%%%%%%%%%%%%%%%%%%
	%%%%%%%%%%%
	
	Riemannian manifolds of Ricci curvature bounded below
	are classical research subjects in comparison geometry and geometric analysis.
	Recently, the diverse developments on the \emph{curvature-dimension condition}
	in the sense of Lott, Sturm and Villani have shed new light on this theory.
	The curvature-dimension condition $\mathrm{CD}(K,N)$
	is a synthetic notion of lower Ricci curvature bounds for metric measure spaces.
	The parameters $K$ and $N$ are usually regarded as
	``a lower bound of the Ricci curvature'' and ``an upper bound of the dimension'',
	respectively.
	Thus $N$ is sometimes called the \emph{effective dimension}.
	The roles of $K$ and $N$ are better understood when we consider
	a weighted Riemannian manifold $(M,g,m)$,
	a Riemannian manifold of dimension $n$ equipped with an arbitrary smooth measure $m$.
	Then the Ricci curvature is modified into the \emph{weighted Ricci curvature}
	$\mathrm{Ric}_N$ involving a parameter $N$,
	and $\mathrm{CD}(K,N)$ is equivalent to $\mathrm{Ric}_N \ge K$
	(see \cite{vRS,StI,StII,LV} as well as \cite{Vi} for $N \in [n,\infty]$, \cite{Oneg,Oneedle} for $N \le 0$,
	and also \cite{Oint} for the Finsler analogue).
	
	For $N \in [n,\infty]$, the weighted Ricci curvature $\mathrm{Ric}_N$
	(also called the \emph{Bakry--\'Emery--Ricci curvature})
	has been intensively studied by, for instance, Bakry and his collaborators
	in the framework of the \emph{$\Gamma$-calculus} (see \cite{BGL}).
	Recently it turned out that there is a rich theory also for $N \in (-\infty,1]$,
	though this range seems strange due to the above interpretation
	of $N$ as an upper dimension bound.
	Among others, various Poincar\'e-type inequalities (\cite{KM}),
	the curvature-dimension condition (\cite{Oneg,Oneedle}),
	isoperimetric inequalities (\cite{Mil,Kl,Oneedle}),
	and the splitting theorem (\cite{WY}) were studied for $N<0$ or $N \le 1$. 
	
	The aim of this article is to contribute to the study
	of the structure of Riemannian manifolds with $\mathrm{Ric}_N \ge K$ for $K>0$ and $N<0$.
	In Section~\ref{sc:1dim}
	we analyze two $1$-dimensional examples $M_1,M_2$ of $\mathrm{Ric}_N \equiv K$
	(appearing in \cite{KM,Mil}) having different natures.
	The first example $M_1$ has finite volume and enjoys only the exponential concentration,
	in particular, $M_1$ does not satisfy the logarithmic Sobolev inequality. 
	For $N<-1$, this space attains the minimum of the first nonzero eigenvalue of the (weighted) Laplacian
	(derived from the Bochner inequality, see Proposition~\ref{pr:sgap}).
	The second example $M_2$ has infinite volume
	and reveals the difficulty to obtain a volume comparison under $\mathrm{Ric}_N \ge K$.
	
	In Section~\ref{sc:sharp}, we investigate
	when the minimum of the first nonzero eigenvalue is attained under the condition $\mathrm{Ric}_N \ge K>0$ with $N<-1$,
	as a counterpart to the Obata rigidity theorem (\cite{Ob}).
	Our main theorem (Theorem~\ref{th:main}) asserts that, when $\dim M \ge 2$,
	we have a warped product splitting
	$M \cong \mathbb{R} \times_{\cosh(\sqrt{K/(1-N)}t)} \Sigma^{n-1}$ of hyperbolic nature.
	Moreover, $\Sigma$ enjoys $\mathrm{Ric}_{N-1} \ge K(2-N)/(1-N)$.
	This warped product splitting should be compared with Cheng--Zhou's theorem in the case of $N = \infty$,
	where the sharp spectral gap forces the space to isometrically split off a $1$-dimensional Gaussian space
	(see \cite{CZ} and Theorem~\ref{th:CZ} for details). Our proof considers the equality case
	in Bochner inequality, that is related
	to the dimension free version in \cite{CZ} but requires a further discussion due to the fact that the Hessian of the eigenfunction does not vanish (Lemma \ref{lm:eigen}). Therefore we have only the warped product rather than the isometric product in \cite{CZ}. We stress that in our setting of $N < 0$, it is interesting to have a hyperbolic structure even when the curvature is positive.
	\medskip
	
	\textit{Acknowledgements}.
	I would like to thank my supervisor, Professor Shin-ichi Ohta,
	for the kind guidance, encouragement and advice he has provided
	throughout my time working on this paper. I also would like to thank Professor Emanuel Milman for the
valuable comments on the earlier version of the paper, especially on the case of $N \in [-1,0)$.
	
	\section{Preliminaries}%%%%%%%%%%%%%%%%%
	%%%%%%%%%%%
	
	\subsection{Weighted Riemannian manifolds}%%%%%%%%%%%
	%%%%%%%%%%%
	
	A weighted Riemannian manifold $(M,g,m)$ will be a pair of
	a complete, connected, boundaryless manifold $M$
	equipped with a Riemannian metric $g$ and a measure $m = e^{-\psi}\mathrm{vol}_{g}$,
	where $\psi \in C^{\infty}(M)$ and $\mathrm{vol}_{g}$ is the standard volume measure on $(M,g)$.
	On $(M,g,m)$, we define the weighted Ricci curvature as follows:
	
	\begin{definition}[Weighted Ricci curvature]\label{df:wRic}
		Given a unit vector $v \in U_{x}M$ and $N \in (-\infty,0) \cup [n,\infty]$,
		the \emph{weighted Ricci curvature} $\mathrm{Ric}_N(v)$ is defined by
		\begin{enumerate}[(1)]
			\item $\mathrm{Ric}_{N}(v) :=\mathrm{Ric}_g(v) +\mathrm{Hess}\,\psi(v,v)
			-\displaystyle\frac{\langle \nabla \psi(x),v\rangle^2}{N-n}$ for $N \in (-\infty,0) \cup (n,\infty)$;
			
			\item $\mathrm{Ric}_{n}(v) :=\mathrm{Ric}_g(v) +\mathrm{Hess}\,\psi(v,v)$
			if $\langle \nabla\psi(x),v\rangle = 0$, and $\mathrm{Ric}_n(v):=-\infty$ otherwise;				
			
			\item $\mathrm{Ric}_{\infty}(v) :=\mathrm{Ric}_g(v) +\mathrm{Hess}\,\psi(v,v)$,
		\end{enumerate}
		where $n=\dim M$ and $\mathrm{Ric}_g$ denotes the Ricci curvature of $(M,g)$.
		The parameter $N$ is sometimes called the \emph{effective dimension}.
		We also define $\mathrm{Ric}_{N}(cv) :=c^{2}\mathrm{Ric}_{N}(v)$ for $c \ge 0$. 
	\end{definition}
	
	Note that if $\psi$ is constant then the weighted Ricci curvature coincides with
	$\mathrm{Ric}_g(v)$ for all $N$.
	When $\mathrm{Ric}_{N}(v) \geq K$ holds for some $K\in \mathbb{R}$ and all unit vectors $v \in TM$,
	we will write $\mathrm{Ric}_{N}\geq K$.
	By definition,
	\[ \mathrm{Ric}_n(v) \le \mathrm{Ric}_N(v) \le \mathrm{Ric}_{\infty}(v) \le \mathrm{Ric}_{N'}(v) \]
	holds for $n \le N<\infty$ and $-\infty<N'<0$,
	and $\mathrm{Ric}_N(v)$ is non-decreasing in $N$ in the ranges $(-\infty,0)$ and $[n,\infty]$.
	
	\begin{remark}
		The weighted Ricci curvature for $N \in [n,\infty]$ has been intensively and extensively investigated,
		see \cite{Qi,BGL} for instance.
		The study for the negative effective dimension $N<0$ is more recent.
		One can find in \cite{Mil,Kl,Oneedle} the isoperimetric inequality,
		in \cite{WY} the Cheeger--Gromoll type splitting theorem,
		and in \cite{Oneg,Oneedle} the curvature-dimension condition in this context.
	\end{remark}
	
	We also define the weighted Laplacian with respect to $m$.
	
	\begin{definition}[Weighted Laplacian]\label{df:Lap}
		The \emph{weighted Laplacian} (also called the \emph{Witten Laplacian})
		of $u \in C^\infty(M)$ is defined as follows:
		\[ \Delta_{m}u := \Delta u - \langle \nabla u,\nabla\psi \rangle. \]
	\end{definition}
	
	Notice that the Green formula (the integration by parts formula)
	\[ \int _{M} u\Delta_{m}v \,dm =-\int _{M}\langle \nabla u,\nabla v\rangle \,dm =\int _{M} v\Delta_{m} u\,dm \]
	holds provided $u$ or $v$ belongs to $C_c^{\infty}(M)$
	(smooth functions with compact supports) or $H^1_0(M)$.
	
	\subsection{Bochner inequality and eigenvalues of Laplacian}%%%%%%%%%
	%%%%%%%%%%%
	
	Recall that the \emph{Bochner--Weitzenb\"ock formula}
	associated with the weighted Ricci curvature $\mathrm{Ric}_{\infty}$ and the weighted Laplacian $\Delta_m$
	holds as follows:
	For $u \in C^{\infty}(M)$, we have
	\begin{equation}\label{eq:BW}
	\Delta_{m}\bigg(\frac{|\nabla u|^{2}}{2} \bigg) -\langle \nabla\Delta_{m}u,\nabla u\rangle
	=\mathrm{Ric}_\infty(\nabla u) +\|\mathrm{Hess}\, u\|_{HS}^{2},
	\end{equation}
	where $\|\cdot\|_{HS}$ stands for the Hilbert--Schmidt norm (with respect to $g$).
	As a corollary we obtain the \emph{Bochner inequality} for $N \in (-\infty,0) \cup [n,\infty]$
	(the case of negative effective dimension was studied independently in \cite{KM,Oneg}).
	We give an outline of the proof for later use.
	
	\begin{theorem}[Bochner inequality]\label{th:Boch}
		For $N \in (-\infty,0) \cup [n,\infty]$ and any $u \in C^{\infty}(M)$, we have
		\[ \Delta_{m} \bigg( \frac{|\nabla u|^{2}}{2} \bigg) -\langle \nabla\Delta_{m}u,\nabla u\rangle
		\geq \mathrm{Ric}_{N}(\nabla u) + \frac{(\Delta_{m}u)^{2}}{N}. \]
	\end{theorem}
	
	\begin{proof}
		Fix $x \in M$ and let $B$ the matrix representation of $\mathrm{Hess}\, u(x)$ for some orthonormal basis.
		Then we find
		\[ \|\mathrm{Hess}\, u\|_{HS}^{2} =\text{tr}(B^2) \geq \frac{(\mathrm{tr}\, B)^2}{n} =\frac{(\Delta u)^2}{n}. \]
		Putting $p=\Delta_{m}u$ and $q=\langle \nabla u,\nabla\psi \rangle$, we have
		\begin{align*}
		\frac{(\Delta u)^2}{n}
		&= \frac{(p+q)^2}{n} =\frac{p^2}{N} -\frac{q^2}{N-n}
		+\frac{N(N-n)}{n}\bigg(\frac{p}{N} + \frac{q}{N-n}\bigg)^2 \\
		&\geq \frac{p^2}{N} -\frac{q^2}{N-n}
		=\frac{(\Delta_{m}u)^2}{N} - \frac{{\langle \nabla u,\nabla\psi \rangle^2}}{N-n}.
		\end{align*}
		Combining these with the Bochner--Weitzenb\"ock formula \eqref{eq:BW} we obtain the Bochner inequality.
	\end{proof}
	
	We remark that $N(N-n) \ge 0$ used in the proof fails for $N \in (0,n)$.
	The Bochner inequality has quite rich applications in geometric analysis.
	In this article we are interested in the first nonzero eigenvalue of the weighted Laplacian,
	a generalization of the \emph{Lichnerowicz inequality}.
	The case of the negative effective dimension $N<0$ was studied in \cite{KM,Oneg}.
	Here we give a proof for completeness.
	
	\begin{proposition}[First nonzero eigenvalue]\label{pr:sgap}
		Let $(M,g,m)$ be a complete weighted Riemannian manifold satisfying $\mathrm{Ric}_{N} \ge K$
		for some $K>0$ and $N<0$, and assume $m(M)<\infty$.
		Then the first nonzero eigenvalue of the nonnegative operator $-\Delta_{m}$
		is bounded from below by $KN/(N-1)$.
	\end{proposition}
	
	\begin{proof}
		Let $u \in H^1_0 \cap C^{\infty}(M)$ be a nonconstant eigenfunction of $\Delta_m$
		with $\Delta_m u=-\lambda u$, $\lambda>0$.
		We deduce from the Bochner inequality (Theorem~\ref{th:Boch})
		and the condition $\mathrm{Ric}_N \ge K$ that
		\[ \Delta_{m} \bigg( \frac{|\nabla u|^{2}}{2} \bigg)
		\ge \langle \nabla\Delta_{m}u,\nabla u\rangle
		+K|\nabla u|^2 +\frac{(\Delta_{m}u)^{2}}{N}. \]
		Since constant functions belong to $H^1_0(M)$ thanks to the hypothesis $m(M)<\infty$,
		the integration of the above inequality yields
		\[ 0 \ge \bigg( \frac{1}{N}-1 \bigg) \int_M (\Delta_m u)^2 \,dm +K\int_M |\nabla u|^2 \,dm. \]
		We again use the integration by parts to see
		\[ \int_M (\Delta_m u)^2 \,dm = -\lambda \int_M u \Delta_m u \,dm
		=\lambda \int_M |\nabla u|^2 \,dm. \]
		Therefore we have
		\[ \frac{1-N}{N}\lambda +K \le 0, \]
		which shows the claim $\lambda \ge KN/(N-1)$.
	\end{proof}
	
	It was observed in \cite[\S 3.2]{KM} and \cite[Theorem~6.1]{Mil}
	that the estimate $\lambda \ge KN/(N-1)$ is sharp for $N \in (-\infty,-1]$
	and, somewhat surprisingly, \emph{not sharp} for $N<0$ close to $0$.
	See Section~\ref{sc:sharp} for a further discussion,
	where we discuss the rigidity for the constant $KN/(N-1)$ in the case $N<-1$.
	The first nonzero eigenvalue of the weighted Laplacian is related to the concentration of measures.
	When $m(M)=1$, we define the \emph{concentration function} of $(M,g,m)$ by
	\[ \alpha(r) :=\sup \big\{ 1-m\big( B(A,r) \big) \,|\, A \subset M:\text{Borel},\, m(A) \ge 1/2 \big\}. \]
	See \cite[Theorem~3.1]{Le} for the following corollary.
	
	\begin{corollary}[Exponential concentration]\label{cr:conc}
		Let $(M,g,m)$ be a compact weighted Riemannian manifold satisfying $\mathrm{Ric}_{N} \geq K$
		for some $K>0$ and $N<0$, and assume $m(M)=1$.
		Then $(M,g,m)$ satisfies the \emph{exponential concentration}
		$\alpha(r) \le e^{-\sqrt{(KN)/(N-1)}r/3}$.
	\end{corollary}
	
	\section{Two examples of $1$-dimensional spaces of constant curvature}\label{sc:1dim}%%%%%%%
	%%%%%%%%%%%
	
	In this section we analyze two examples of $1$-dimensional spaces such that
	$\mathrm{Ric}_N \equiv K$ for some $K>0$ and $N<0$.
	These examples will be helpful to understand the general picture
	of the spaces satisfying $\mathrm{Ric}_N \ge K>0$.
	In particular, we shall give negative answers to some naive guesses.
	
	The first example is the following:
	
	\begin{example}\label{ex:Mil}
		For $K>0$ and $N<0$, the space
		\[ M_1:= \bigg( \mathbb{R},|\cdot|,\cosh\bigg( {\sqrt{\frac{K}{1-N}}}x\bigg)^{N-1} \,dx \bigg), \]
		where $|\cdot|$ denotes the canonical distance structure,
		satisfies $\mathrm{Ric}_N \equiv K$.
	\end{example}
	
	This is a model space in Milman's isoperimetric inequality (see Case~1 of \cite[Corollary~1.4]{Mil}).
	Notice that the total volume is finite since
	\begin{align*}
	\int_{\mathbb{R}} \cosh\bigg( {\sqrt{\frac{K}{1-N}}}x\bigg)^{N-1} \,dx
	&\le 2\int_0^{\infty} \big( e^{\sqrt{K/(1-N)}x} \big)^{N-1} \,dx \\
	&= 2 \int_0^{\infty} e^{-\sqrt{K(1-N)}x} \,dx <\infty.
	\end{align*}
	In order to see the curvature identity,
	we observe that the weight function is given by
	\[ \psi(x) =(1-N)\log \bigg( {\cosh} \sqrt{\frac{K}{1-N}}x \bigg), \]
	and
	\begin{align*}
	\psi'(x) &=\sqrt{K(1-N)} \tanh\bigg( \sqrt{{\frac{K}{1-N}}}x \bigg), \\
	\psi''(x) &=K -K\tanh \bigg( \sqrt{\frac{K}{1-N}}x \bigg)^2.
	\end{align*}
	Since the Ricci curvature vanishes in the $1$-dimensional case,
	we have the curvature identity
	\begin{align*}
	&\mathrm{Ric}_{N} \bigg( \frac{\partial}{\partial x}\Big|_x \bigg)
	= \psi''(x) +\frac{\psi'(x)^2}{1-N} \\
	&= K -K \tanh \bigg( \sqrt{\frac{K}{1-N}}x \bigg)^2
	+\frac{K(1-N)}{1-N} \tanh \bigg( \sqrt{\frac{K}{1-N}}x \bigg)^2 \\
	&= K.
	\end{align*}
	
	In \cite{KM} it was observed that this space attains the minimum value of the first nonzero eigenvalue in respect of Proposition~\ref{pr:sgap}.
	
	\begin{lemma}[Minimal first nonzero eigenvalue]\label{lm:sharp}
		For $N < -1$, the first nonzero eigenvalue of $-\Delta_m$ in the space $M_1$ in Example~$\ref{ex:Mil}$
		coincides with $KN/(N-1)$.
	\end{lemma}
	
	\begin{proof}
		We shall show that the function
		\[ u(x) =\sinh \bigg( \sqrt{\frac{K}{1-N}}x \bigg) \]
		gives the minimal first nonzero eigenvalue.
		Notice that $u\in L^2(M_1)$ by the assumption $N < -1$, and
		\[ u'(x) =\sqrt{\frac{K}{1-N}} \cosh \bigg( \sqrt{\frac{K}{1-N}}x \bigg), \qquad
		u''(x) =\frac{K}{1-N} u. \]
		Thus the weighted Laplacian of $u$ is calculated as
		\begin{align*}
		\Delta_{m}u(x) &=u''(x) -u'(x) \psi'(x)
		= \frac{K}{1-N}u(x) -K \sinh \bigg( \sqrt{\frac{K}{1-N}}x \bigg) \\
		&= \frac{K}{1-N}u(x) -Ku(x) =-\frac{KN}{N-1}u(x).
		\end{align*}
	\end{proof}
	
	\begin{remark}[Failure of log-Sobolev inequality]\label{rm:LSI}
		For $N \in [n,\infty]$, the bound $\mathrm{Ric}_N \ge K>0$ is known to imply
		the \emph{logarithmic Sobolev inequality}
		(see, for example, \cite{BGL} for the proofs based on the $\Gamma$-calculus):
		\begin{equation}\label{eq:LSI}
		\int_{\{f>0\}} f \log f \,dm \le \frac{N-1}{2KN} \int_{\{f>0\}} \frac{|\nabla f|^2}{f} \,dm
		\end{equation}
		for nonnegative functions $f \in H^1(M)$ with $\int_M f \,dm=1$.
		It is well known that the logarithmic Sobolev inequality implies
		the \emph{normal concentration}, see \cite[Theorem~5.3]{Le}.
		One might expect that the logarithmic Sobolev inequality \eqref{eq:LSI} has a counterpart for $N<0$
		similarly to the spectral gap.
		This is, however, not the case since $M_1$ (normalized as $m(M_1)=1$)
        given in Example~\ref{ex:Mil} enjoys only the exponential concentration
        (consider $A=[0,\infty)$).
        See \cite[Proposition~6.4]{Mil} for a detailed estimate of the concentration function.
	\end{remark}
	
	Our second example is as follows.
	
	\begin{example}\label{ex:flat}
		For $K>0$ and $N<0$, the space
		\[ M_2:= (\mathbb{R},|\cdot|,e^{-\sqrt{K(1-N)}x} \,dx) \]
		satisfies $\mathrm{Ric}_N \equiv K$.
	\end{example}
	
	This space also appeared in Case~2 of \cite[Corollary~1.4]{Mil}.
	The weight function is the linear function $\psi(x)=\sqrt{K(1-N)}x$ and
	the curvature identity is straightforward:
	\[ \mathrm{Ric}_{N} \bigg( \frac{\partial}{\partial x}\Big|_x \bigg)
	=\frac{\psi'(x)^2}{1-N} =K. \]
	
	We stress that $M_2$ has very different natures from $M_1$ in Example~\ref{ex:Mil}.
	Recall that, when $N=\infty$,
	$\mathrm{Ric}_{\infty} \ge K>0$ implies that the measure has the Gaussian decay
	and the total volume is finite (\cite[Theorem~4.26]{StI}).
	From Example~\ref{ex:Mil} one may expect that $\mathrm{Ric}_N \ge K>0$ for $N<0$
	still implies the exponential decay, however,
	it is not the case by the second example $M_2$ above.
	We refer to \cite{Sa2} for a related work on weighted Riemannian manifolds with boundaries,
	where a volume comparison for regions spreading from the boundary was established for $N \le 1$.
	
	We also remark that, different from the case of $N=\infty$,
	taking products will destroy the curvature bound for $N<0$
	due to the nonlinearity of the term $\langle \nabla\psi,v \rangle^2$ in $\psi$
	in the definition of $\mathrm{Ric}_N(v)$.
	For instance,
	\[ (\mathbb{R}^n,|\cdot|,e^{-x^1-x^2- \cdots -x^n} \,dx^1 dx^2 \cdots dx^n) \]
	satisfies only $\mathrm{Ric}_N \ge 0$
	(consider $v$ in the kernel of $d\psi$).
	
	\section{First nonzero eigenvalue of Laplacian}\label{sc:sharp}%%%%%%%%%%%%%
	%%%%%%%%%%%%%
	
	As we saw in Proposition~\ref{pr:sgap},
	the first nonzero eigenvalue of $-\Delta_{m}$
	on a weighted Riemannian manifold satisfying $\mathrm{Ric}_{N} \ge K>0$
	is bounded from below by $KN/(N-1)$ (or $K$ if $N = \infty$),
	which is sharp for $N \le -1$. 
	In this section we consider the rigidity problem on when the minimum value is attained.
	In the unweighted ($N=n$) case,
	the classical Obata theorem \cite{Ob} asserts the following.
	
	\begin{theorem}[Obata rigidity theorem]\label{th:Obata}
		Let $(M,g)$ be an $n$-dimensional Riemannian manifold
		with $\mathrm{Ric}_{g} \geq K >0$.
		Then the first nonzero eigenvalue of $-\Delta$ coincides with $Kn/(n-1)$
		if and only if $M$ is isometric to the $n$-dimensional sphere $\mathbb{S}^{n}$
		with radius $\sqrt{(n-1)/K}$.
	\end{theorem}
	
	The case of $N \in (n,\infty)$ turns out void as follows.
	
	\begin{theorem}[Ketterer, Kuwada]\label{th:KK}
		Let $(M,g,m)$ be a weighted Riemannian manifold of dimension $n$
		with $\mathrm{Ric}_N \geq K >0$ for some $N \in [n,\infty)$.
		Then the first nonzero eigenvalue of $-\Delta_m$ coincides with $KN/(N-1)$
		if and only if $N=n$, $m=c \cdot \mathrm{vol}_g$ for some constant $c>0$
		and $M$ is isometric to the $n$-dimensional sphere $\mathbb{S}^{n}$ with radius $\sqrt{(n-1)/K}$.
	\end{theorem}
	
	Precisely, it is shown in \cite{Ke} in the general framework of RCD-spaces
	(metric measure spaces satisfying the \emph{Riemannian curvature-dimension condition}) that
	the existence of an eigenfunction for the eigenvalue $KN/(N-1)$
	implies the maximal diameter $\mathrm{diam}(M)=\pi\sqrt{(N-1)/K}$
	in the Bonnet--Myers type theorem.
	This is, however, achieved only when $N=n$ by Kuwada's result in \cite{Ku}.
	If we admit singularities, then the sharp spectral gap is attained for spherical suspensions
	as discussed in \cite{Ke} for RCD-spaces (see also \cite[\S 3.2]{KM}).
	
	Next, in the case of $N=\infty$, an interesting rigidity theorem was established in \cite{CZ}.
	
	\begin{theorem}[Cheng--Zhou]\label{th:CZ}
		Let $(M,g,m)$ be an $n$-dimensional weighted Riemannian manifold
		with $\mathrm{Ric}_{\infty} \geq K >0$.
		If the first nonzero eigenvalue of $-\Delta_m$ coincides with $K$,
		then $M$ is isometric to the product space $\Sigma^{n-1} \times \mathbb{R}$
		as weighted Riemannian manifolds,
		where $\Sigma^{n-1}$ is an $(n-1)$-dimensional manifold with $\mathrm{Ric}_{\infty} \ge K$
		and $\mathbb{R}$ is equipped with the Gaussian measure $e^{-Kx^2/2}dx$.
	\end{theorem}
	
	The proof of Theorem~\ref{th:CZ} has a certain similarity with
	the \emph{Cheeger--Gromoll splitting theorem}.
	The role of the Busemann function in \cite{CG} is replaced with
	the eigenfunction in \cite{CZ}.
	Theorem~\ref{th:CZ} was recently generalized to RCD-spaces in \cite{GKKO}.
	
	Now we consider the case of $N<0$.
	We begin with an important equation for an eigenfunction.
	
	\begin{lemma}\label{lm:eigen}
		Let $(M,g,m)$ be a complete weighted Riemannian manifold satisfying $\mathrm{Ric}_{N} \ge K$
		for some $N<0$ and $K>0$, and $m(M)<\infty$.
		Suppose that the first nonzero eigenvalue of $-\Delta_m$ coincides with $KN/(N-1)$.
		Then the eigenfunction $u$ of $KN/(N-1)$ necessarily satisfies
		\begin{equation}\label{eq:Hessu}
		\mathrm{Hess}\, u =-\frac{Ku}{N-1} \cdot g
		\end{equation}
		as operators $TM \times TM \longrightarrow \mathbb{R}$.
	\end{lemma}
	
	\begin{proof}
		Let $u$ be an eigenfunction of the eigenvalue $KN/(N-1)$, namely
		\[ \Delta_m u =-\frac{KN}{N-1} u. \]
		Tracing back to the proof of the lower bound of first nonzero eigenvalue,
		the Bochner inequality for $u$ necessarily becomes equality.
		Thus we find, from the proof of Theorem~\ref{th:Boch},
		\begin{align}
		&\mathrm{Hess}\, u=f \cdot g \quad \text{for some function}\ f:M \longrightarrow \mathbb{R};
		\label{eq:eigen1}\\
		&\frac{\Delta_{m}u}{N} +\frac{\langle \nabla u,\nabla\psi\rangle}{N-n} \equiv 0.
		\label{eq:eigen2}
		\end{align}
		On the one hand, we observe from the former equation \eqref{eq:eigen1} that $\Delta u=nf$.
		On the other hand, the latter \eqref{eq:eigen2} yields
		\[ \Delta u =\Delta_m u +\langle \nabla u,\nabla\psi \rangle
		=\frac{n}{N} \Delta_m u =-\frac{Kn}{N-1} u. \]
		Therefore we have $f =-Ku/(N-1)$ and complete the proof.
	\end{proof}
	
	\begin{theorem}\label{th:main}
		Let $(M,g,m)$ be a complete weighted Riemannian manifold of dimension $n$
		satisfying $\mathrm{Ric}_{N} \ge K$ for some $N<-1$ and $K>0$,
		and $m(M)<\infty$.
		Assume that the first nonzero eigenvalue of $-\Delta_m$ coincides with $KN/(N-1)$.
		\begin{enumerate}[{\rm (i)}]
			\item If $n \geq 2$, then $M$ is isometric to the warped product
			\[ \mathbb{R} \times_{\cosh(\sqrt{K/(1-N)}t)} \Sigma
			=\bigg( \mathbb{R} \times \Sigma,
			 dt^2 +\cosh^2\bigg( \sqrt{\frac{K}{1-N}}t \bigg) \cdot g_{\Sigma} \bigg) \]
			and the measure $m$ is written through the isometry as
			\[ m(dtdx) =\cosh^{N-1}\bigg( \sqrt{\frac{K}{1-N}}t \bigg) \,dt \,m_{\Sigma}(dx), \]
			where $(\Sigma,g_{\Sigma},m_{\Sigma})$ is an $(n-1)$-dimensional weighted Riemannian manifold
			satisfying $\mathrm{Ric}_{N-1} \ge K(2-N)/(1-N)$.
			
			\item If $n=1$,
			then $(M,g,m)$ is isometric to the model space $M_1$ in Example~$\ref{ex:Mil}$
			up to a constant multiplication of the measure.
		\end{enumerate}
	\end{theorem}
	
	\begin{proof}
		(i)
		By rescaling the Riemannian metric, we can assume that $K=1-N$.
		Let $u \in H^1_0 \cap C^{\infty}(M)$ be an eigenfunction satisfying $\Delta_{m}u =Nu$.
		Along any geodesic $\gamma: I \longrightarrow M$ for an interval $I \subset \mathbb{R}$,
		it follows from Lemma~\ref{lm:eigen} that $(u \circ \gamma)'' =(u \circ \gamma)$.
		This implies
		\begin{equation}\label{eq:ugamma}
		u \circ \gamma(t) =u \circ \gamma(0) \cdot \cosh t +(u \circ \gamma)'(0) \cdot \sinh t.
		\end{equation}
		
		Put $\Sigma:=u^{-1}(0)$ (which is nonempty since $\int_M u \,dm=0$).
		We deduce from \eqref{eq:ugamma} that, for $x,y \in \Sigma$,
		every geodesic connecting them is included in $\Sigma$.
		Hence $\Sigma$ is totally geodesic and the mean curvature of $\Sigma$ is identically zero.
		Lemma~\ref{lm:eigen} also shows that, for any smooth vector field $V$
		along a geodesic $\gamma$ included in $\Sigma$,
		\[ \langle \nabla_{\dot{\gamma}} (\nabla u \circ \gamma),V \rangle
		=\mathrm{Hess}\, u(\dot{\gamma},V) =u \cdot \langle \dot{\gamma},V \rangle \equiv 0. \]
		Hence $\nabla u \circ \gamma$ is a parallel vector field
		and $|\nabla u|$ is constant on $\Sigma$.
		We will normalize $u$ so as to satisfy $|\nabla u| \equiv 1$ on $\Sigma$.
		
		Fix $x \in \Sigma$ and consider the geodesic $\gamma_x:[0,\infty) \longrightarrow M$
		such that $\dot{\gamma}_x(0)=\nabla u(x)$.
		Then \eqref{eq:ugamma} and our normalization $|\nabla u|(x)=1$ yield
		\begin{equation}\label{th:lap}
		(u \circ \gamma_{x})(t) = \sinh t, \qquad t \in \mathbb{R}.
		\end{equation}
		For any unit speed geodesic $\eta : [0,s] \longrightarrow M$ from $x$ to $\gamma_{x}(t)$, \eqref{th:lap} and \eqref{eq:ugamma} imply $\sinh t = u(\eta(s)) \leq \sinh s$. Since $\sinh t$ is monotone increasing, we have $s \geq t$ and $\gamma_x$ is globally minimizing. 
		Denoting the distance function from $\Sigma$ by $\rho_{\Sigma}$,
		we obtain from \eqref{th:lap} that $\rho_{\Sigma} =|\mathrm{arcsinh}(u)|$ and,
		on $u^{-1}((0,\infty))$,
		\[ \Delta u =\Delta (\sinh \rho_{\Sigma})
		=\cosh \rho_{\Sigma} \cdot \Delta \rho_{\Sigma} +\sinh \rho_{\Sigma} \cdot |\nabla \rho_{\Sigma}|^2. \]
		Since $|\nabla\rho_{\Sigma}|=1$ and $\Delta u =nu =n \sinh \rho_{\Sigma}$,
		we have for $t \ge 0$
		\begin{equation}\label{th:lapequal}
		\Delta\rho_{\Sigma} \big( \gamma_{x}(t) \big)
		=(n-1)\tanh \rho_{\Sigma} \big( \gamma_x(t) \big) =(n-1)\tanh t.
		\end{equation}
		
		We shall compare the Laplacian comparison inequality \eqref{th:lapequal}
		with the (unweighted) Ricci curvature along $\gamma_x$.
		Using \eqref{th:lap} and \eqref{eq:eigen2}, we have
		
		\begin{align*}
		(\psi \circ \gamma_x)'(t)
		&= \langle \nabla \psi \circ \gamma_x, \dot{\gamma}_x \rangle (t)
		=\frac{1}{\cosh t} \langle \nabla \psi,\nabla u \rangle \big( \gamma_x(t) \big) \\
		&=\frac{n-N}{\cosh t} u \big( \gamma_x(t) \big) =(n-N) \tanh t.
		\end{align*}
		Hence
		\begin{equation}\label{th:splitweight}
		\psi \circ \gamma_x(t) = (n-N) \log(\cosh t) + \psi(x)
		\end{equation}
		and \begin{align*}
		\textrm{Ric}\big( \dot{\gamma}_{x}(t) \big)
		&= \textrm{Ric}_{N} \big( \dot{\gamma}_{x}(t) \big)
		-(\psi \circ \gamma_x)''(t) -\frac{(\psi \circ \gamma_x)'(t)^2}{n-N} \\
		&= \textrm{Ric}_{N} \big( \dot{\gamma}_{x}(t) \big)
		-\frac{n-N}{\cosh^2 t} -(n-N)\tanh^2 t \\
		& =\textrm{Ric}_{N} \big( \dot{\gamma}_{x}(t) \big) - (n-N) \ge 1-n.
		\end{align*}
		This curvature bound together with the minimality of $\Sigma$ implies that
		\begin{equation}\label{th:Ric} \Delta\rho_{\Sigma} \big( \gamma_{x}(t) \big) \le (n-1)\tanh t \end{equation}
		by \cite[Corollary~2.44]{Ka} (see also \cite{HK}, \cite[Theorem~4.3]{Sa1}).
		Then, since we have equality \eqref{th:lapequal},
		the rigidity theorem for the Laplacian comparison shows that
		$u^{-1}([0,\infty))$ is isometric to the warped product
		\[ [0,\infty) \times_{\cosh} \Sigma
		:=\big( [0,\infty) \times \Sigma, dt^2 +\cosh^2 t \cdot g_{\Sigma} \big), \]
		where $g_{\Sigma}$ is the Riemannian metric of $\Sigma$ induced from $g$ of $M$
		(see the proof of \cite[Theorem~1.8]{Sa1} for details). Here we give a sketch of the proof for the isometry along \cite{Sa1}.
		
		For a regular point $x\in\Sigma$ of $u$, let $T_{x}^{\perp}\Sigma$ be the orthogonal complement of $T_{x}\Sigma$. Choose an orthonormal basis $\{e_{x,i}\}_{i=1}^{n-1}$ of $T_{x}\Sigma$ and for $i = 1,...,n-1$,let $Y_{x,i}$ be the $\Sigma$-Jacobi field along $\gamma_{x}$ such that $Y(0)=e_{x,i}$ and $Y'_{x,i}(0)=-A_{\nabla u(x)}e_{x,i}$, where $A_{v}: T_{x}\Sigma\longrightarrow T_{x}\Sigma$ be the shape operator of the tangent vector $v$. Since the mean curvature on $\Sigma$ is $0$ and $\textrm{Ric}_{N} \big( \dot{\gamma}_{x}(t) \big) \geq 1-n$, we have $\Delta\rho_{\Sigma} \big( \gamma_{x}(t) \big) \le (n-1)\tanh t$ as in \eqref{th:Ric} and equality holds if and only if $Y_{x,i}(t) = \cosh(t)E_{x,i}(t)$ where $E_{x,i}$ are the parallel vector fields along $\gamma_{x}$ with $E_{x,i}(0) = e_{x,i}$. Then the map $\phi: [0,\infty)\times\Sigma\longrightarrow M$ defined by $\phi(t,x) := \gamma_{x}(t)$ provides the desired isometry between $[0,\infty) \times_{\cosh} \Sigma$ and $M$.

		By a similar discussion on $u^{-1}((-\infty,0])$
		we can extend the isometry to
		\[ M \cong \mathbb{R} \times_{\cosh} \Sigma
		=\big( \mathbb{R} \times \Sigma, dt^2 +\cosh^2 t \cdot g_{\Sigma} \big). \]
		Concerning the splitting of the measure,
		let $\Sigma$ be equipped with the measure
		$m_{\Sigma}:=e^{-\psi|_{\Sigma}} \mathrm{vol}_{g_{\Sigma}}$.
		By \eqref{th:splitweight} and the structure of the warped product, we see that the measure $m$ is written through the isometry as
		\begin{align*}
		m(dtdx) &=\cosh^{N-n} t \cdot e^{-\psi(x)} \cdot
		 \big( \cosh^{n-1} t \,dt \,\mathrm{vol}_{g_{\Sigma}}(dx) \big) \\
		&= \cosh^{N-1} t \,dt \,m_{\Sigma}(dx).
		\end{align*}
		With this expression one can see that the eigenfunction $u$ is in $L^2(M)$ when $N<-1$
		(but not in $L^2(M)$ for $N \in [-1,0)$).
		
		Finally, for any unit vector $v \in T_x \Sigma$ at $x \in \Sigma$,
		the sectional curvature of the plane $\dot{\gamma}_x(0) \wedge v$
		coincides with $-1$ (see \cite[Proposition~42(2) in \S 7]{ON}).
		Therefore we have, on the weighted Riemannian manifold
		$(\Sigma,g_{\Sigma},m_{\Sigma})$,
		\begin{align*}
		\mathrm{Ric}^{\Sigma}_{N-1}(v)
		&= \mathrm{Ric}^{\Sigma}(v) +\mathrm{Hess}^{\Sigma}\, \psi(v,v)
		-\frac{\langle \nabla^{\Sigma} \psi,v \rangle^2}{(N-1)-(n-1)} \\
		&= \mathrm{Ric}^M(v)+1 +\mathrm{Hess}^M\, \psi(v,v)
		-\frac{\langle \nabla^M \psi,v \rangle^2}{N-n} \\
		&=\mathrm{Ric}^M_N(v) +1 \ge 2-N.
		\end{align*}

		(ii)
		Fix an eigenfunction $u$ of the eigenvalue $KN/(N-1)$,
		take $x \in M$ with $u(x)=0$ and choose a unit speed geodesic
		$\gamma:\mathbb{R} \longrightarrow M$ with $\gamma(0)=x$.
		It follows from Lemma~\ref{lm:eigen} that
		\[ u \circ \gamma(t) =(u \circ \gamma)'(0) \cdot \sinh\bigg( \sqrt{\frac{K}{1-N}}t \bigg). \]
		Since $u$ is not constant, we have $(u \circ \gamma)'(0) \neq 0$
		and hence $u$ is injective.
		Therefore $M$ is isometric to $\mathbb{R}$
		and we will identify them via $\gamma$ as $\gamma(t)=t$.
		
		Denote by $\psi:\mathbb{R} \longrightarrow \mathbb{R}$ the weight function,
		namely $m=e^{-\psi} \,dt$.
		Then it follows from
		\[ -\frac{KN}{N-1} u =\Delta_m u =u'' -u' \psi' \]
		that
		\begin{align*}
		\psi'(t) &= \frac{1}{u'(t)} \bigg( u''(t) +\frac{KN}{N-1}u(t) \bigg) \\
		&=\sqrt{\frac{1-N}{K}} \bigg( \frac{K}{1-N} +\frac{KN}{N-1} \bigg)
		\tanh \bigg( \sqrt{\frac{K}{1-N}}t \bigg) \\
		&= \sqrt{K(1-N)} \tanh \bigg( \sqrt{\frac{K}{1-N}}t \bigg).
		\end{align*}
		Integrating this we have
		\[ \psi(t) =(1-N) \log \bigg( {\cosh} \bigg( \sqrt{\frac{K}{1-N}}t \bigg) \bigg) +\psi(0), \]
		which implies $m=e^{-\psi(0)} \cosh(\sqrt{K/(1-N)} t)^{N-1} \,dt$ as desired.
	\end{proof}
	
	The following corollary is a byproduct of the proof above
	(see \cite{KM,Mil} for the $1$-dimensional case).
	
	\begin{corollary}\label{cr:N>-1}
	Let $(M,g,m)$ be a complete weighted Riemannian manifold of dimension $n$
		satisfying $\mathrm{Ric}_{N} \ge K$ for some $N \in [-1,0)$ and $K>0$,
		and $m(M)<\infty$.
	Then $KN/(N-1)$ cannot be an eigenvalue of $-\Delta_m$.
	\end{corollary}
	
	\begin{remark}\label{rm:Sigma}
		The curvature bound of $\Sigma$ implies
		the estimate of the first nonzero eigenvalue $\lambda_1(\Sigma)$ as
		\[ \lambda_1(\Sigma) \ge \frac{K(2-N)}{1-N} \frac{N-1}{N-2} =K. \]
		This is better than the estimate $\lambda_1(M) \ge KN/(N-1)$ for $M$.
	\end{remark}
		
	%In \cite{Sa1}, Sakurai also used the Laplacian comparison of \cite{HK} to yield the splitting theorem on Riemannian manifolds with boundary satisfying Ricci curvature bound and mean curvature bound on boundary conditions.
	
	We close the article with a concrete example of the splitting phenomenon
	described in Theorem~\ref{th:main}(i) (with $K=1-N$).
	
	\begin{example}\label{ex:hyp}
		Let $(\mathbb{H}^2,g)$ be the hyperbolic plane of constant sectional curvature $-1$.
		There are smooth functions $u$ and $\psi$ satisfying
		\[ \Delta_m u =Nu, \qquad \mathrm{Ric}_N \equiv 1-N, \]
		where $m=e^{-\psi} \mathrm{vol}_g$.
	\end{example}
	
	Let us give the precise expressions of $u$ and $\psi$ in the upper half-plane model:
	\[ (\mathbb{H}^2,g) = \bigg( \mathbb{R} \times (0,\infty),\frac{dx^2+dy^2}{y^2} \bigg). \]
	By the consideration in the proof of Theorem~\ref{th:main}(i),
	we can expect that the function $u$ is written by using the distance function from the $y$-axis
	($\Sigma=u^{-1}(0)$ coincides with the $y$-axis).
	One can explicitly calculate it as
	\begin{align*}
	u(x,y) &= \pm \sinh \Big( d\big( (x,y),(0,\sqrt{x^2+y^2}) \big) \Big) \\
	&= \pm \sinh \bigg( \mathrm{arccosh}\bigg( 1+\frac{x^2+(y-\sqrt{x^2+y^2})^2}{2y\sqrt{x^2+y^2}} \bigg) \bigg) \\
	&= \pm \sqrt{\frac{x^2+y^2}{y^2}-1} =\frac{x}{y},
	\end{align*}
	where we choose the sign $+$ for $x>0$ and $-$ for $x<0$,
	and we used the equation $\sinh(\mathrm{arccosh}\,t)=\sqrt{t^2-1}$ for $t \ge 0$.
	The Christoffel symbols are readily calculated as
	\[ \Gamma^1_{11} =\Gamma^1_{22} =\Gamma^2_{12} \equiv 0, \quad
	\Gamma^1_{12}(x,y) =\Gamma^2_{22}(x,y) =-\frac{1}{y}, \quad
	\Gamma^2_{11}(x,y) =\frac{1}{y}. \]
	Using these we find
	\begin{align*}
	(\mathrm{Hess}\, u)_{(x,y)}
	&= \begin{pmatrix} 0 & -y^{-2} \\ -y^{-2} & 2xy^{-3} \end{pmatrix}
	-\frac{1}{y} \begin{pmatrix} 0 & -y^{-1} \\ -y^{-1} & 0 \end{pmatrix}
	+\frac{x}{y^2} \begin{pmatrix} y^{-1} & 0 \\ 0 & -y^{-1} \end{pmatrix} \\
	&= \frac{x}{y} \cdot \frac{1}{y^2} \begin{pmatrix} 1&0 \\ 0&1 \end{pmatrix}
	=u(x,y) \cdot g_{(x,y)}.
	\end{align*}
	
	Next we consider the weight function.
	Let
	\begin{align*}
	\psi_{1}(x,y)
	&:=(2-N) \log \Big[ \cosh \Big( d\big( (x,y),(0,\sqrt{x^2+y^2}) \big) \Big) \Big] \\
	&=(2-N) \log \bigg( \frac{\sqrt{x^2+y^2}}{y} \bigg),
	\end{align*}
	and 
	\begin{align*}
	\psi_{2}(x,y) := -(2-N) \log ( \sqrt{x^2+y^2}).
	\end{align*}
	Notice that $\psi_1$ naturally appears in respect of \eqref{eq:eigen2},
	and $\psi_2$ is employed to improve the convexity in the directions perpendicular to $\nabla u$.
	Indeed, we have $\langle \nabla u,\nabla\psi_{2} \rangle = 0$ and
	\begin{align*}
	\langle \nabla u,\nabla\psi_1 \rangle(x,y)
	&=(2-N)y^2 \bigg( \frac{1}{y} \cdot \frac{x}{x^2+y^2}
	-\frac{x}{y^2} \cdot \bigg( \frac{y}{x^2+y^2}-\frac{1}{y} \bigg) \bigg) \\
	&= (2-N) \frac{x}{y} =(2-N)u(x,y).
	\end{align*}
	Our weighted function will be $\psi = \psi_{1} + \psi_{2}$, namely:
	\begin{align*}
	\psi(x,y) = -(2-N) \log y.
	\end{align*}
	The above calculations yields
	\[ \Delta_m u =\Delta u -\langle \nabla u,\nabla\psi \rangle
	=2u -(2-N)u =Nu. \]
	
	We finally calculate $\mathrm{Ric}_N$ of this example.
	To this end, we observe
	\[ \frac{(\mathrm{Hess}\, \psi)_{(x,y)}}{2-N}
	= \begin{pmatrix} y^{-2} &
	0 \medskip\\
	0 &
	0 \end{pmatrix}, \qquad
	\frac{(d\psi \otimes d\psi)_{(x,y)}}{(2-N)^2}
	= \begin{pmatrix} 0 &
	0 \medskip\\
	0 &
	y^{-2} \end{pmatrix}. \]
	Therefore we obtain
	\begin{align*}
	(\mathrm{Ric}_N)_{(x,y)}
	&=-\frac{1}{y^2} \begin{pmatrix} 1&0 \\ 0&1 \end{pmatrix}
	+(\mathrm{Hess}\, \psi)_{(x,y)}
	+\frac{(d\psi \otimes d\psi)_{(x,y)}}{2-N} \\
	&= \frac{1-N}{y^2} \begin{pmatrix} 1&0 \\ 0&1 \end{pmatrix}.
	\end{align*}

\end{document}